\documentclass[12pt]{article}

\usepackage[affil-it]{authblk}
\usepackage[left=1.25in,top=1.25in,right=1.25in,bottom=1.25in,head=1.25in]{geometry}
\usepackage{amsfonts,amsmath,amssymb,amsthm}
\usepackage{verbatim,float,color,dsfont}
\usepackage{graphicx,subfigure,url}
\usepackage{natbib}
\usepackage{hyperref}

\usepackage{color}

\renewcommand{\refname}{References}
\newcommand{\email}[1]{\href{mailto:#1}{#1}}

\newcommand{\R}{{\mathbb R}}

\newcommand{\Z}{{\mathbb Z}}

\newcommand{\M}{{\mathbb M}}
\newcommand{\X}{{\mathbb X}}
\newcommand{\Y}{{\mathbb Y}}

\theoremstyle{plain}

\newtheorem{theorem}{Theorem}[section]

\newtheorem{lemma}[theorem]{Lemma}

\newtheorem{example}[theorem]{Example}
\newtheorem{definition}[theorem]{Definition}

\newcommand{\figref}[1]{Figure~\ref{#1}}

\newcommand{\secref}[1]{Section~\ref{#1}}

\title{\textbf{On the Bootstrap for \\ Persistence Diagrams and Landscapes}}

\author[1]{Fr\'ed\'eric Chazal}
\author[2]{Brittany Terese Fasy}
\author[3]{Fabrizio Lecci}
\author[3]{\\Alessandro Rinaldo}
\author[4]{Aarti Singh}
\author[3]{Larry Wasserman}

\affil[1]{INRIA Saclay}
\affil[2]{Computer Science Department, Tulane University}
\affil[3]{Department of Statistics, Carnegie Mellon University}
\affil[4]{Machine Learning Department, Carnegie Mellon University}

\date{\email{topstat@stat.cmu.edu} \\ \text{} \\ \text{} \\
January 22, 2014
}

\begin{document}
\maketitle

\begin{abstract}
\noindent Persistent homology probes topological properties 
from point clouds and functions. 
By looking at multiple scales simultaneously, 
one can record the births and deaths of \emph{topological features} 
as the scale varies. 
In this paper we use a statistical technique, the empirical bootstrap,
to separate \emph{topological signal} from \emph{topological noise}. 
In particular, 
we derive confidence sets for persistence diagrams and confidence 
bands for persistence landscapes.
\end{abstract}

\clearpage
\section* {Introduction}
Persistent homology is a method for studying the
homology at multiple {scales} simultaneously.
Given a manifold
$\X$ embedded in a metric space $\mathbb{Y}$,
we consider a probability
density function $p \colon \mathbb{Y} \to \R$, 
defined over $\mathbb{Y}$ but concentrated
around $\X$; that is, the density is positive for a small
neighborhood around $\X$ and very small over $\mathbb{Y} \, \backslash \, \X$.
For the right scale parameter~$t$, the superlevel set $p^{-1}([t,\infty))$
captures the homology of $\X$.  The problem, however, is that~$t$ is not known
a priori.  Persistent homology
quantifies the topological changes of the superlevel sets with a
multiset of points in the extended plane; we call this
multiset the persistence diagram, and denote it by $\mathcal{P}$.
Another way to represent the information contained in a 
persistence diagram is with the \emph{landscape} function
$\mathcal{L} \colon \R \to \R$, which can be thought of as
a functional summary of $\mathcal{P}$; we define these 
concepts in \secref{ss:topology}.

Computationally, it may be difficult to compute $\mathcal{P}$
or $\mathcal{L}$
directly.  Instead, we assume that $p$ corresponds to 
a probability distribution $P$, from which we can sample.
Given a sample of size $n$, we create an estimate
of the probability density function $p_n$ using a
kernel density estimate.  As $n$ increases,
$p_n$ approaches the true probability density.  Given
$n$ large enough, we compute the persistence diagram
$\mathcal{P}_n$ and the landscape $\mathcal{L}_n$
corresponding to $p_n$.  

Sometimes knowing the estimate of a persistence
diagram or landscape is not enough.
The bigger question is: How close is the estimated
persistence diagram or landscape to the true one?
We answer this question by constructing
a \emph{confidence set
for persistence diagrams} and a \emph{confidence
band for persistence landscapes}.

A \emph{$(1-\alpha)$-confidence interval} for a
parameter $\theta$ is an interval $[a,b]$ such that
the probability $\mathbb{P}(\theta \in [a,b])$ is at
least $1-\alpha$.  In our setting,
we desire to find a confidence set for a persistence
diagram $\mathcal{P}$.  To do so, we compute
an estimated diagram $\widehat{\mathcal{P}}$ and
and interval $[0,c]$ such that the bottleneck
distance between $\mathcal{P}$ and 
$\widehat{\mathcal{P}}$ is contained in $[0,c]$
with probability $1-\alpha$.  That is, we find
a metric ball containing $\mathcal{P}$ with
high~probability.

In this paper, we present the bootstrap, a method 
for computing confidence intervals, and we apply it
to persistence diagrams and landscapes.
After briefly reviewing the necessary concepts from 
computational topology, we give the general
technique of boot\-strapping in statistics in 
\secref{ss:bootstrap}.  In
\secref{sec:boot4pers}, we apply the bootstrap
to persistence diagrams and landscapes, providing
a few examples of these confidence intervals.
We conclude in \secref{sec:discussion} with a
discussion of our ongoing research and open questions.

%%%%%%%%%%%%%%%%%%%%%%%%%%%%%%%%%%%%%%%%%%%%%%%%%%%%%%%%%%%%%%%%%%%%%%%%
% Section: Background
%          (1) VERY Brief review of TDA
%          (2) Introduction to the Bootstrap
%%%%%%%%%%%%%%%%%%%%%%%%%%%%%%%%%%%%%%%%%%%%%%%%%%%%%%%%%%%%%%%%%%%%%%%%
\section{Background}
Before presenting our results, we review the necessary definitions and
theorems from persistent homology.  
Then, we present the bootstrap.  Due to space constraints, we cover the
basics and provide references for a more detailed description.
%As this paper is intended to introduce the bootstrap to computational
%topologists, we focus this section on explaining the bootstrap.

\subsection{Persistence Diagrams and Landscapes}\label{ss:topology}
Let $\Y$ be a metric space, for example. let $\Y$ be 
a compact subspace of $\R^D$.
Suppose we have a probability density function $p \colon \Y \to \R$
concentrated in a neighborhood of a manifold $\X \subseteq \Y$.
Persistent homology monitors the evolution of the generators
of the homology groups of $p^{-1}([t,\infty))$, the superlevel sets of $p$,
and assigns to each generator of these groups a birth time 
(or scale) $b$ and a death time $d$.
The persistence diagram~$\mathcal{P}$ 
records each pair $(b,d)$ as
the point $(\frac{b+d}{2}, \frac{b - d}{2})$; that is,
the $x$-coordinate is the mid-life of the homological feature
and the $y$-coordinate is the half-life or half of the persistence of 
the feature.\footnote{In this paper, we focus on the persistent homology of
the superlevel set filtration of a density function.  Thus,
the birth time $b$ is greater than the death time $d$.}
We refer the reader to~\cite{edelsharer} for a more complete
introduction to persistent~homology.

Let ${\cal D}_T$ be the space of positive, countable, 
$T$-bounded persistence diagrams; that is,
for each point~$(x,y) = (\frac{b+d}{2}, \frac{b - d}{2}) \in {\cal P}$, 
we have $0 \leq d \leq b \leq T$ and there
are a countable number of points for which $y > 0$.
We note here that each point on the line $x=0$ is included in
the persistence diagram $\mathcal P$ with infinite multiplicity.
Letting $W_{\infty}(\mathcal{P}_1, \mathcal{P}_2)$ denote the bottleneck
distance between diagrams $\mathcal{P}_1$ and $\mathcal{P}_2$,
the space $({\cal D}, W_{\infty})$ is a
metric space.  We then have the following stability result
from~\cite{stability} and generalized in~\cite{chazal2012structure}:
\begin{theorem}[Stability Theorem]
\label{th:stability}
Let $\M$ be finitely triangulable.
Let $f,g \colon \M \to \R$ be two continuous functions.
Then, the corresponding persistence diagrams $\mathcal{P}_f$ and $\mathcal{P}_g$
are well defined, and
$
W_\infty({\cal P}_f,{\cal P}_g) \leq  \Vert f - g\Vert_\infty.
$
\end{theorem}

\cite{bubenik2012statistical} introduced another representation
called the persistence landscape, which is in one-to-one
correspondence with persistence diagrams.
A persistence landscape
is a continuous, piecewise linear 
function~\mbox{$\mathcal{L} \colon  \Z^{+} \times \R \to \R$.}  
To define the persistence landscape function, we replace
each persistence point $p = (x,y)= \left( \frac{b+d}{2}, \frac{b-d}{2} \right)$
with the triangle function 
\begin{equation*}
 t_p(z) =
 \begin{cases}
  z-x+y & z \in [x-y, x] \\
  x+y-z & z \in (x,  x+y] \\
  0 & \text{otherwise}
 \end{cases}
 =
 \begin{cases}
  z-d & z \in [d, \frac{b+d}{2}] \\
  b-z & z \in (\frac{b+d}{2}, b] \\
  0 & \text{otherwise}.
 \end{cases}
\end{equation*}
Notice that $p$ is itself on the graph of $t_p(z)$.
We obtain
an arrangement of curves by overlaying the graphs of 
the functions~\mbox{$\{ t_p(z) \}_{p \in \mathcal{P}}$}; see~\figref{fig:triangles}.  
\begin{figure}
 \centering
 \includegraphics[height=1.4in]{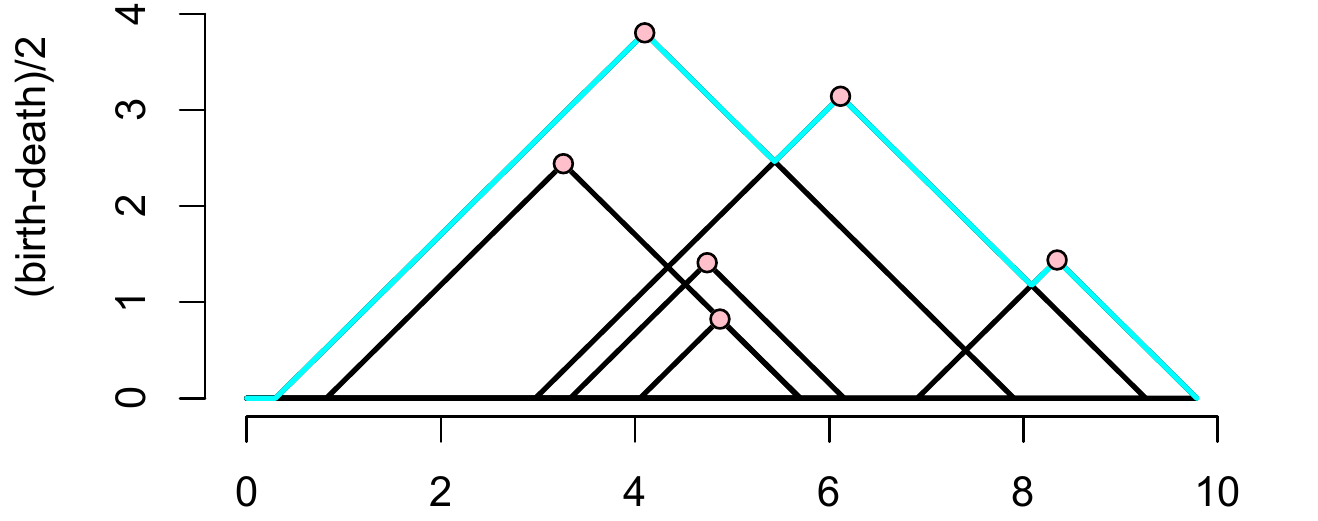}
 \caption{The pink circles are the points in a persistence diagram.
          The cyan curve is the landscape $\mathcal{L}(1,\cdot)$.}
 \label{fig:triangles}
\end{figure}
The persistence landscape is defined formally as
a walk through this arrangement:
\begin{equation}\label{eq:landscape}
 \mathcal{L}_{\mathcal{P}}(k,z) = \underset{p \in \mathcal{P}}{\text{kmax}} ~ t_p(z),
\end{equation}
where kmax is the $k$th maximum value in the set; in particular,
$1$max is the usual maximum~function.
Observe that $\mathcal{L}_{\mathcal{P}}(k,z)$ is $1$-Lipschitz.
For the ease of exposition, we will focus on $k=1$ in this paper,
using $\mathcal{L}(z) = \mathcal{L}_{\mathcal{P}}(1,z)$.  However,
the ideas we present in~\secref{ss:landscapeboot} hold for~$k > 1$. 
Our definition of the persistence landscape 
is equivalent to the original definition given
in \cite{bubenik2012statistical}.

\subsection{The Standard Bootstrap}\label{ss:bootstrap}
Introduced in \cite{efron1979bootstrap}, the bootstrap is a general method 
for estimating standard errors 
and for computing confidence intervals.
We focus on the latter in this paper, but refer the 
interested reader to \cite{efron2001empirical,davison1997bootstrap}, and \cite{van2000asymptotic}
for more details on the versatility of the bootstrap.

Let $X_1, \dots, X_n$ be independent and identically distributed random variables taking values in the measure space $(\mathbb{X}, \mathcal{X}, P)$. 
Suppose we are interested in estimating the real-valued parameter~$\theta$ corresponding to the distribution $P$ of the observation. 
We estimate $\theta$ using the statistic $\hat \theta=g(X_1, \dots, X_n)$, which is
some function of the data. For example, $\theta$ and $\hat \theta$ could be the population mean and the sample mean, respectively. 
The distribution of the difference~\mbox{$\hat \theta-\theta$} contains all the information that we need to construct a confidence interval of level $1-\alpha$ for $\theta$;
that is, an interval $[a,b]$ depending on the data $X_1, \dots, X_n$ such that
$
\mathbb{P}(\theta \in [a,b]) ~\geq~ 1-\alpha.
$ 
If we knew the cumulative distribution $F$ of $\hat \theta-\theta$, then the quantiles
$F^{-1}(1- \alpha/2)$ and $F^{-1}(\alpha/2)$ can be computed.
Furthermore, setting
$a= \hat \theta - F^{-1}(1-\alpha/2)$ and $b= \hat \theta -F^{-1}(\alpha/2)$,
we immediately obtain a $(1-\alpha)$-confidence interval for $\theta$:
\begin{align*}
\mathbb{P}(\theta \in [a,b]) 
%&= \mathbb{P}\left(\hat \theta - F^{-1}\left(1-\frac{\alpha}{2} \right) 
%              ~\leq~ \theta ~\leq~ \hat \theta -F^{-1}\left(\frac{\alpha}{2}\right) \right) \\
&=  \mathbb{P}\left(F^{-1}\left(\frac{\alpha}{2}\right)  
              ~\leq~ \hat \theta - \theta ~\leq~ F^{-1}\left(1-\frac{\alpha}{2} \right)  \right)=1-\alpha.
\end{align*}
Unfortunately, the distribution of $\hat \theta-\theta$ depends on the unknown distribution $P$.

In the first step in the bootstrap procedure, we approximate the unknown $P$ 
with the empirical measure $P_n$ that puts mass $1/n$ at each $X_i$ in the sample. Let $X_1^*, \dots, X_n^*$ be a sample of size $n$ from $P_n$. Equivalently,
we can think of drawing $X_1^*, \dots, X_n^*$ from $X_1, \dots, X_n$ with replacement.
We estimate the distribution $F( r)$ with the distribution $\widehat F ( r)= P_n(\hat \theta^* - \hat \theta \leq r)$, 
where $\hat \theta^*=g(X_1^*, \dots, X_n^*)$.

The distribution $\widehat F$ is still not analytically computable, 
but can be approximated by simulation: for large $B$, obtain $B$ 
different values of $\hat \theta^\ast$ and approximate $\widehat F( r)$, and 
hence $F(r)$, with
\mbox{$
\widetilde F( r)= \frac{1}{B} \sum_{j=1}^B I(\hat \theta_{j}^* - \hat \theta \leq r).
$}
Since the quantiles of $\widetilde F$ approximate the quantiles of $F$, we define the estimated confidence interval as
\begin{equation}
\label{eq:interval}
C_n=\left[\hat \theta - \widetilde F_n^{-1}(1- \alpha/2)\; , \; \hat \theta - \widetilde F_n^{-1}(\alpha/2)  \right].
\end{equation}

\noindent In summary, the standard bootstrap procedure is: 
%\\ \\
%\fbox{\parbox{\textwidth}{
\begin{enumerate}
\item Compute the estimate $\hat \theta= g(X_1, \dots, X_n)$.
\item Draw $X_1^*, \dots, X_n^*$ from $P_n$ and compute $\hat \theta^*=g(X_1^*, \dots, X_n^*)$. 
\item Repeat the previous step $B$ times to obtain $\hat\theta_{1}^*, \dots, \hat \theta_{B}^*$.
\item Compute the quantiles of $\widetilde F$ and construct the confidence interval $C_n$.
\end{enumerate}
%}}
%\\
%\text{}
%\\
There are two sources of error in the Bootstrap procedure. We first approximate $F$ with $\widehat F$ 
and then we estimate $\widehat F$ by simulation. The second error can be made arbitrarily small, 
by choosing $B$ large enough. Therefore,
this error is usually ignored in the theory of the bootstrap.
\\Formally, one has to show that 
$
\sup_r \left | \widetilde F ( r) - F( r) \right | \stackrel{P}{\to}0 \; ,
$
 which implies that the confidence interval $C_n$, defined in \eqref{eq:interval}, is \textit{asymptotically consistent} at level $1-\alpha$; that is, 
$
\liminf_{n \to \infty} \mathbb{P}(\theta \in C_n) \geq 1-\alpha.
$

\subsection{The Bootstrap Empirical Process}\label{ss:epboot}
When a random variable is a function rather than a real value,
the bootstrap procedure described above can be used to construct
a confidence interval for the function evaluated at a particular
element of the domain.
Instead, we use the \emph{bootstrap empirical process}, which can be 
used to find a confidence band for a function $h(t)$; that is, we find
a pair of functions $a(t)$ and $b(t)$ such that
the probability that $ h(t) \in [a(t), b(t)]$  for all $t$
is at least $1-\alpha$.
We describe this technique below, but refer the reader to 
\cite{van1996weak} and \cite{kosorok2008introduction} for more details.

An \textit{empirical process} is a stochastic process based on a random sample. Let $X_1, \dots, X_n$ be 
independent and identically distributed random variables taking values in the measure 
space $(\mathbb{X}, \mathcal{X}, P)$. For a measurable function $f: \mathbb{X} \to \mathbb{R}$,
we denote $Pf= \int f dP$ and $P_nf= \int f dP_n= n^{-1} \sum_{i=1}^n f(X_i)$. 
By the law of large numbers $P_nf$ converges almost surely to $Pf$.
Given a class $\mathcal{F}$ of measurable functions, we define the empirical process $\mathbb{G}_n$ indexed by $\mathcal{F}$ as 
$$
\left\{ \mathbb{G}_n f \right\}_{f \in \mathcal{F}}= \left\{ \sqrt{n} (P_nf-Pf) \right\}_{f \in \mathcal{F}}.
$$
\begin{example} 
If $\mathcal{F}=\{I(x\leq t ) \}_{t \in \R}$,
then $\{P_nf\}_{ f \in \mathcal{F}}= \{n^{-1} \sum_{i=1}^n I(X_i \leq t)\}_{ t \in \mathbb{R}}$, 
which is the empirical distribution function seen as a stochastic process indexed by~$t$.
Furthermore, we have
$\{\mathbb{G}_n f \}_{f \in \mathcal{F}} = 
     \{ n^{-1/2} \sum_{i=1}^n I(X_i \leq t ) - P(X_i \leq t) \}_{t \in \mathbb{R}} $.
\end{example}

\begin{definition}
\label{def:donsker}
A class $\mathcal{F}$ of measurable functions $f: \mathbb{X} \to \mathbb{R}$ is 
called $P$-Donsker if the process $\{ \mathbb{G}_n f\}_{f \in \mathcal{F}}$ 
converges in distribution to a limit process in the space~$\ell^\infty(\mathcal{F})$, 
where $\ell^\infty(\mathcal{F})$ is the collection of all bounded functions 
$f: \X \to \mathbb{R}$. \\
The limit process is a Gaussian process $\mathbb{G}$ with 
zero mean and covariance function \mbox{$\text{E } \mathbb{G}f\mathbb{G}g := Pfg -PfPg$};
this process is known as a Brownian Bridge.
\end{definition}

Let $P_n^*f = n^{-1} \sum_{i=1}^n f(X_i^*)$ where $\{X_1^*, \dots, X_n^*\}$ is a bootstrap sample from $P_n$, the measure that puts mass $1/n$ on each element of the sample $\{X_1, \dots, X_n \}$.
The \textit{bootstrap empirical process} $\mathbb{G}_n^*$ indexed by $\mathcal{F}$ is defined as 
$$
\{ \mathbb{G}_n^* f \}_{f \in \mathcal{F}} = \{ \sqrt{n} (P_n^*f-P_nf) \}_{f \in \mathcal{F}}.
$$

\begin{theorem} [Theorem 2.4 in \cite{gine1990bootstrapping}] 
\label{th:bootProcess}
$\mathcal{F}$ is $P$-Donsker if and only if
$\mathbb{G}_n^*$ converges in distribution to $\mathbb{G}$ in $\ell^\infty({\mathcal{F}})$.
\end{theorem}
In words, Theorem \ref{th:bootProcess} states that $\mathcal{F}$ is $P$-Donsker if and only if
 the bootstrap empirical process converges in distribution to the limit process $\mathbb{G}$ given in Definition \ref{def:donsker}.
Suppose we are interested in constructing a confidence band of level $1-\alpha$ for $\{Pf\}_{ f \in \mathcal{F} }$, where $\mathcal{F}$ is $P$-Donsker. Let $\hat \theta=\sup_{f \in \mathcal{F}} \left | \mathbb{G}_nf \right |$. 
We proceed as follows: 
%\\ \\
%\fbox{\parbox{\textwidth}{
\begin{enumerate}
\item Draw $X_1^*, \dots, X_n^*$ from $P_n$ and compute $\hat \theta^* = \sup_{f \in \mathcal{F}} \left | \mathbb{G}_n^*f \right |$. 
\item Repeat the previous step $B$ times to obtain $\hat\theta_{1}^*, \dots, \hat \theta_{B}^*$.
\item Compute 
$
q_{\alpha} = \inf \left \{q: \frac{1}{B} \sum_{j=1}^B I(\hat \theta_j^* \geq q) \leq \alpha \right\}.
$
\item For $f \in \mathcal{F}$ define the confidence band 
$
C_n(f)= \left[P_nf- \frac{q_{\alpha}}{\sqrt{n}} \,, \, P_nf+ \frac{q_{\alpha}}{\sqrt{n}}\right] .
$
\end{enumerate}
%}}
%\\
%\text{} \\ 
A consequence of Theorem \ref{th:bootProcess} is that, for large $n$ and $B$, the interval $[0,q_{\alpha}]$ has coverage $1-\alpha$ for $\hat \theta$ and the band $C_n(f)_{f \in \mathcal{F}}$ has coverage $1-\alpha$ for $\{Pf\}_{f \in \mathcal{F}}$.

%Note that
%\begin{align*}
%\mathbb{P}( Pf \in C_n(f) , \forall f \in \mathcal{F}  )  &= \mathbb{P}\left( \left |  \sqrt{n} (P_nf-Pf)\right |  \leq q_{1-\alpha} , \forall f \in \mathcal{F}  \right) \\
%&= \mathbb{P}\left( \left | \mathbb{G}_nf\right |  \leq q_{1-\alpha} , \forall f \in \mathcal{F} \right) \\
%&= \mathbb{P}\left( \sup_{f \in \mathcal{F}} \left | \mathbb{G}_nf\right |  \leq  q_{1-\alpha}  \right) \\
%&= \mathbb{P}( \theta \in [0, q_{1-\alpha}]),
%\end{align*}
%and, according to Theorem \ref{th:bootProcess}, for large $n$ and $B$, $[0, q_{1-\alpha}]$ has coverage $1-\alpha$ \todo{more details needed}.
%Therefore $\{C_n(f), f \in \mathcal{F}\}$ is a confidence band of level $1-\alpha$ for $\{Pf ~|~ f \in \mathcal{F}\}$.

%%%%%%%%%%%%%%%%%%%%%%%%%%%%%%%%%%%%%%%%%%%%%%%%%%%%%%%%%%%%%%%%%%%%%%%%
% Section: Applying the Bootstrap
%          (1) to Persistence Diagrams
%          (2) to Landscapes
%%%%%%%%%%%%%%%%%%%%%%%%%%%%%%%%%%%%%%%%%%%%%%%%%%%%%%%%%%%%%%%%%%%%%%%%
\section{Applications of the Bootstrap}\label{sec:boot4pers}
In this section, we apply the bootstrap from 
the previous section to persistence diagrams, as
well as to persistence landscapes.

\subsection{Persistence Diagrams}\label{ss:dgmbootstrap}
Let $X_1,\ldots, X_n$ be a sample from the distribution $P$, supported on a smooth manifold $\mathbb{X} \subset \mathbb{R}^D$.
Let
$
p_h(x) = \int_{\mathbb{X}} 
\frac{1}{h^D} K\left( \frac{||x-u||}{h}\right) dP(u),
$ 
where $K: \mathbb{R} \to \mathbb{R}$ is an integrable function satisfying $\int K(u) du =1$
and $K(u)$ is nonnegative for all $u$;
thus $p_h$ is a probability distribution.
The function $K$ is called a \textit{kernel} and the parameter $h>0$  is its \textit{bandwidth}. 
Then $p_h$ is the density of the probability measure
$P_h$ which is the convolution
$P_h = P\star \mathbb{K}_h$
where
$\mathbb{K}_h(A) = h^{-D} \mathbb{K}(h^{-1} A)$ and
$\mathbb{K}(A) = \int_A K(t) dt$.
$P_h$
is a smoothed version of $P$.

Our target of inference in this section is
${\cal P}_h$, the persistence diagram
of the superlevel sets of $p_h$.
The standard estimator for $p_h$ is the kernel density estimator
$$
\hat p_h(x) =
\frac{1}{n}\sum_{i=1}^n
\frac{1}{h^D}\,K\left(\frac{||x-X_i||}{h}\right);
$$
notice that if $X_i$ are fixed, then $\hat p_h$ is a porbability distribution.
Let $\mathcal{\widehat P}_h$ be the corresponding persistence diagram.
We wish to find a confidence set for $\mathcal{P}_h$, i.e.~, 
an interval $[0,c_n]$ such that
$\limsup_{n\to\infty}\mathbb{P}(W_\infty(\mathcal{\widehat P}_h, \mathcal{P}_h) \in  [0,c_n] ) \geq 1-\alpha$.
From Theorem \ref{th:stability} (Stability),
it suffices to find $c_n$ such that 
$
\limsup_{n\to\infty}\mathbb{P}(\Vert \hat p_h - p_h\Vert_\infty >  c_n) \leq \alpha.
$

To find $c_n$, we use the bootstrap. 
Let $\mathcal{F}=\left\{f_x(u)=\frac{1}{h^D} K\left( \frac{\Vert x-u \Vert}{h} \right) \right\}_{x \in \X}$. 
Using the notation of Section \ref{ss:epboot}, it follows that $Pf_x= p_h(x)$, $P_nf_x=\hat p_h(x)$ and $\hat \theta=\sup_{f_x \in \mathcal{F}}|\mathbb{G}_n f_x| = \sqrt{n} \Vert \hat p_h - p_h \Vert_\infty$. 
The approximated $1-\alpha$ quantile $q_\alpha$ can be obtained through simulation, i.e.,
$q_{\alpha} = \inf \{ q ~:~ \frac1B \sum_{j=1}^B I(\sqrt{n} || \hat{p}_n^j - \hat{p}_n || \geq q) \leq \alpha \}$,
where $p_h^j(x)$ denotes the probability distribution corresponding to the $j^{th}$
bootstrap sample. The following result holds under suitable regularity 
conditions on the kernel $K$ for which $\mathcal{F}$ is Donsker;
see \cite{gine2002rates}.
\begin{theorem}[Lemma 15 in~\cite{arxiv_stathom}] 
We have that
$$ \;\;
\limsup_{n \to \infty} \mathbb{P}\Bigl(\sqrt{n} \Vert \hat p_h - p_h \Vert_\infty >  q_\alpha\Bigr) ~\leq~ \alpha. 
$$
\end{theorem}
%As usual, we approximate $Z_\alpha$ by Monte Carlo.
%Let
%$T = \sqrt{n h^D}||\hat p_h - \hat p_h^*||_\infty$
%be from a bootstrap sample.
%Repeat bootstrap $B$ times
%yielding values
%$T_1,\ldots, T_B$.
%Let
%$$
%\widetilde{Z}_\alpha = \inf \Bigl\{ z:\ \frac{1}{B}\sum_{j=1}^B I( T_j > z) \leq \alpha \Bigr\}.
%$$
%We can ignore the error due to the fact that $B$ is finite
%since this error can be made as small as we like.
% We have emphasized fixed $h$ asymptotics
% since, for topological inference, it is not necessary to let $h\to 0$
% as $n\to\infty$.
% However, it is possible to let $h\to 0$ if one wants.
% Suppose $h\equiv h_n$ and $h\to 0$ as $n\to \infty$.
% We require that
% $n h^D/\log n \to \infty$ as $n\to \infty$.
%It follows from Theorem 3.4
%of~\cite{neumann1998strong}~that
% \begin{equation}
% \mathbb{P}\left( W_\infty(\hat{\cal P}_h,{\cal P}_h) >     \frac{Z_\alpha}{\sqrt{nh^D}}\right) \leq
% \mathbb{P}\Bigl(\sqrt{n h^D}||\hat p_h - p_h||_\infty > Z_\alpha\Bigr) = \alpha +
% \left( \frac{\log n}{n h^D}\right)^{\frac{4+D}{2(2+D)}}.
% \end{equation}
By the Stability Theorem, we conclude:
$ \;\;
\lim_{n \to \infty} \mathbb{P}\left( W_\infty(\widehat{\cal P}_h,{\cal P}_h) >     \frac{q_\alpha}{\sqrt{n}}\right) \leq \alpha.
$

\begin{example}[Torus]\label{ex:torus}
We embed the torus $\mathbb{S}^1 \times \mathbb{S}^1$ in $\R^3$ and we use the 
rejection sampling algorithm of \cite{diaconis2012sampling} ($R=1.5, r=0.8$) to sample 
$10,000$ points uniformly from the torus. 
Then, we compute the persistence diagram $\mathcal{\widehat P}_h$ using
the Gaussian kernel with bandwidth $h=0.25$
and use the bootstrap to construct the $0.95\%$ confidence interval $[0 \, , \, 0.01]$ 
for $W_\infty(\mathcal{\widehat P}_h, \mathcal{P}_h)$;
see Figure \ref{fig:torus}.  Notice that the confidence set
correctly captures the topology of the torus. That is, only the points representing 
real features of the torus are significantly far from the horizontal axis.
\end{example}

\begin{figure}
 \centering
 \includegraphics[height=2.3in]{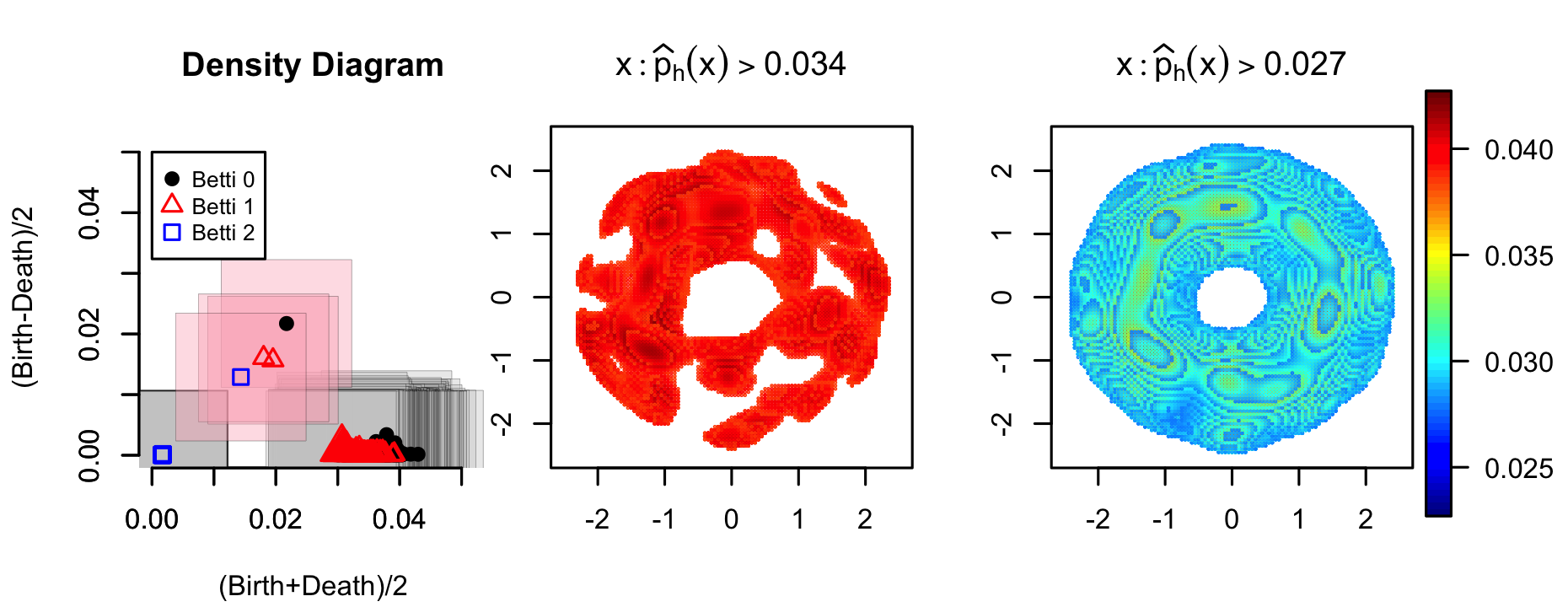}
 \caption{Left: Persistence Diagram of the superlevel sets of 
 a kernel density estimator on the 3D torus described in Example \ref{ex:torus}. 
 The boxes of side $=2 \times 0.01$ around the points represent the 95\% confidence set for $\mathcal{P}_h$. 
 Middle: 2D projection of the superlevel set $\{x: \hat p_h(x)> 0.034\}$. Right: 2D projection of the superlevel set $\{x: \hat p_h(x)> 0.027\}$. }
 \label{fig:torus}
\end{figure}

\subsection{Landscapes}\label{ss:landscapeboot}
Let the diagrams $\mathcal{P}_1, \ldots,\mathcal{P}_n  $ be a sample from
the distribution $P$ over the space of persistence diagrams $\mathcal{D}_T$.
Thus, by definition, we have $x+y \leq T < \infty$ and $0 \leq y \leq T/2$
for all $(x,y) \in \cup_i \mathcal{P}_i$.

Let $\mathcal{L}_1, \ldots, \mathcal{L}_n$ be the landscape functions
corresponding to $\mathcal{P}_1, \dots, \mathcal{P}_n $. That is,  
$\mathcal{L}_i(t) = \mathcal{L}_{\mathcal{P}_i}(1,t)$, as defined in~\eqref{eq:landscape}. 
We define the \textit{mean landscape}
$
\mu(t) = \mathbb{E}_{P}[ \mathcal{L}_{i}(t)],
$ 
and the \textit{empirical mean landscape}
$
\overline{\cal L}_n(t) = \frac{1}{n}\sum_{i=1}^n \mathcal{L}_{i}(t).
$ 
In this section, we show that the process
$
\sqrt{n}(\overline{\cal L}_n(t) - \mu(t))
$ 
converges to a Gaussian process, so that
we may use the procedure given in~\secref{ss:epboot}.

Let
${\cal F} = \{ f_t ~:~ 0\leq t\leq T\}$,
where
$f_t:{\cal D}\to \mathbb{R}$ is defined by
$f_t(\mathcal{P}) = \mathcal{L}_\mathcal{P}(1,t)$.  We note here that
$f_t(\mathcal{P}) = 0$ if $t \notin (0,T)$.
We can now write
$\sqrt{n}(\overline{\cal L}_n(t) - \mu(t))$ as an empirical process indexed by $t \in [0,T]$ :
$$
\sqrt{n}(\overline{\cal L}_n(t) - \mu(t)) =  \sqrt{n} \left( \frac{1}{n}\sum_{i=1}^n \mathcal{L}_i(t) - \mu(t)\right)=\sqrt{n}(P_nf_t - Pf_t)\equiv \mathbb{G}_nf _t.
$$
We note that
the constant function
$F(\mathcal{P}) = T/2$ is a measurable envelope for ${\cal F}$.

Given a probability measure $Q$ over $\mathcal{F}$, let
$\Vert f-g \Vert_{Q,2} = \sqrt{\int |f-g|^2 dQ}$ and let
$N({\cal F},L_2(Q),\varepsilon)$ be the covering number of $\mathcal{F}$, that is,
the size of the smallest $\varepsilon$-net in this~metric.

\begin{lemma}[Theorem 2.5 in \cite{kosorok2008introduction}]
\label{lemma:covering}
Let $\mathcal{F}$ be a class of measurable functions satisfying
$
\int_0^1 \sqrt{\log \sup_Q N(\mathcal{F}, L_2(Q), \varepsilon \Vert F \Vert_{Q,2})} d\varepsilon < \infty \, ,
$
where $F$ is a measurable envelope of $\mathcal{F}$ and the supremum is taken over all finitely discrete probability measures $Q$ with $\Vert F \Vert_{Q,2} >0$. If $PF^2 < \infty$, then $\mathcal{F}$ is $P$-Donsker.
\end{lemma}

\begin{theorem}[Weak Convergence of Landscapes]
Let $\mathbb{G}$ be a Brownian bridge
with covariance function
$
\kappa(t,u) = \int f_t(\mathcal{P})f_u(\mathcal{P}) dP(\mathcal{P}) - \int f_t(\mathcal{P}) dP(\mathcal{P}) \int f_u(\mathcal{P}) dP(\mathcal{P}).
$ 
Then,
$\mathbb{G}_n$ converges in distribution to $\mathbb{G}.$
\end{theorem}

\begin{proof}[Proof]
Since persistence landscapes are $1$-Lipschitz, we have
$
\Vert f_t -f_u \Vert_{Q,2} \leq |t-u|.
$ 
Construct a regular grid
$0\equiv t_0 < t_1 < \cdots < t_N \equiv T$,
where $t_{j+1} - t_{j} = \varepsilon \Vert F \Vert_{Q,2} = \varepsilon \, T/2$.
We claim that
 $\{f_{t_j} : 1\leq j \leq N\}$ is an $(\varepsilon \, T/2)$-net for ${\cal F}$:
choose $f_t \in \cal F$; then there is a $j$ so that $t_j \leq t \leq t_{j+1}$ and
$
\Vert f_{t_{j+1}} -  f_t \Vert_{Q,2} \leq |t_{j+1}-t| \leq |t_{j+1}-t_j|=\varepsilon \, T/2.
$ 
The fact that $\{f_{t_j}: 1\leq j \leq N \}$ is an $(\varepsilon \, T/2)$-net
implies
$
\sup_Q N({\cal F},L_2(Q), \varepsilon \Vert F \Vert_{Q,2}) \leq 2/\varepsilon.
$ 
Hence,
$
\int_0^1 \sqrt{\log \sup_Q N({\cal F},L_2(Q),\varepsilon \Vert F \Vert_2 )}d\varepsilon < \infty.
$
$F={T/2}$ is trivially square-integrable. By Lemma \ref{lemma:covering}, $\mathbb{G}_n$ converges in distribution to $\mathbb{G}$.
\end{proof}

Now that we have shown that $\mathbb{G}_n$ converges to
a Gaussian process, we can follow the procedure outlined 
in~\secref{ss:epboot}.
Let $P_n$ be the empirical measure that puts mass $1/n$
at each diagram $\mathcal{P}_i$.
We draw $\mathcal{P}_1^{\ast}, \ldots \mathcal{P}_n^{\ast}$ from $P_n$ and construct
the corresponding landscapes $\mathcal{L}_1^\ast, \dots, \mathcal{L}_n^\ast$.
Let $\mathcal{\overline L}_n^*$ be the empirical mean and
$\hat{\theta}^{\ast} =$ \mbox{$\sup_{t \in \R} | \sqrt{n} (\mathcal{\overline L}_n^*(t) - \mathcal{\overline L}_n(t)) |$.}
Repeating this $B$ times, we obtain
$\hat{\theta}_1^{\ast}, \ldots \hat{\theta}_B^{\ast}$, and we compute the quantile $q_{\alpha}$.
\begin{theorem}[Confidence Band for Persistent Landscapes]
The interval $C_n(t)$ indexed by $t \in \R$, defined by 
$C_n(t)= \left[ \mathcal{\overline L}_n(t) - \frac{q_{\alpha}}{\sqrt{n}} \; , \; \mathcal{\overline L}_n(t) + \frac{q_{\alpha}}{\sqrt{n}}\right] $,
is a confidence band for $\mu(t)$:
  \begin{equation*}
  \lim_{n \to \infty} \mathbb{P} \left( \mu(t) \in C_n(t) \; \text{ for all } t \right) \geq 1- \alpha.
  \end{equation*}
\end{theorem}

\begin{example}[Circles]\label{ex:circles}
Given the nine circles of radii $0.4$ and $0.3$, shown in \figref{fig:circles},
we obtain a sample $X_1, \ldots, X_{100}$ as follows: first, choose
a circle $C_i$ uniformly at random, then sample a point iid from $C_i$.
Let $\mathcal{P}$ be the (Betti 1) persistence diagram corresponding to the
Rips filtration for the sample, and $\mathcal{L}$ be the landscape corresponding
to $\mathcal{P}$.  \footnote{Note that, since in this example we are using sublevel sets, 
the role of birth and death in the definitions of section \ref{ss:topology} is inverted. The death time $d$ is greater than the birth time $b$.}
We repeat this $50$ times to obtain
diagrams $\mathcal{P}_1, \ldots \mathcal{P}_{50}$ and
landscapes $\mathcal{L}_1, \ldots \mathcal{L}_{50}$.
Then, we use the bootstrap procedure to obtain the quantile $q_{\alpha} = 0.234$. Together with $\mathcal{\overline L}_{50}$, this gives us an
approximated $95\%$
confidence band for $\mu(t)= \mathbb{E}_P(\mathcal{L}_i(t))$. On the right of \figref{fig:circles} we show the empirical mean landscape $\mathcal{\overline L}_{50}$ with the $95\%$ confidence band for $\mu(t)$.
\end{example}
\begin{figure}
 \centering
 \includegraphics[height=2in]{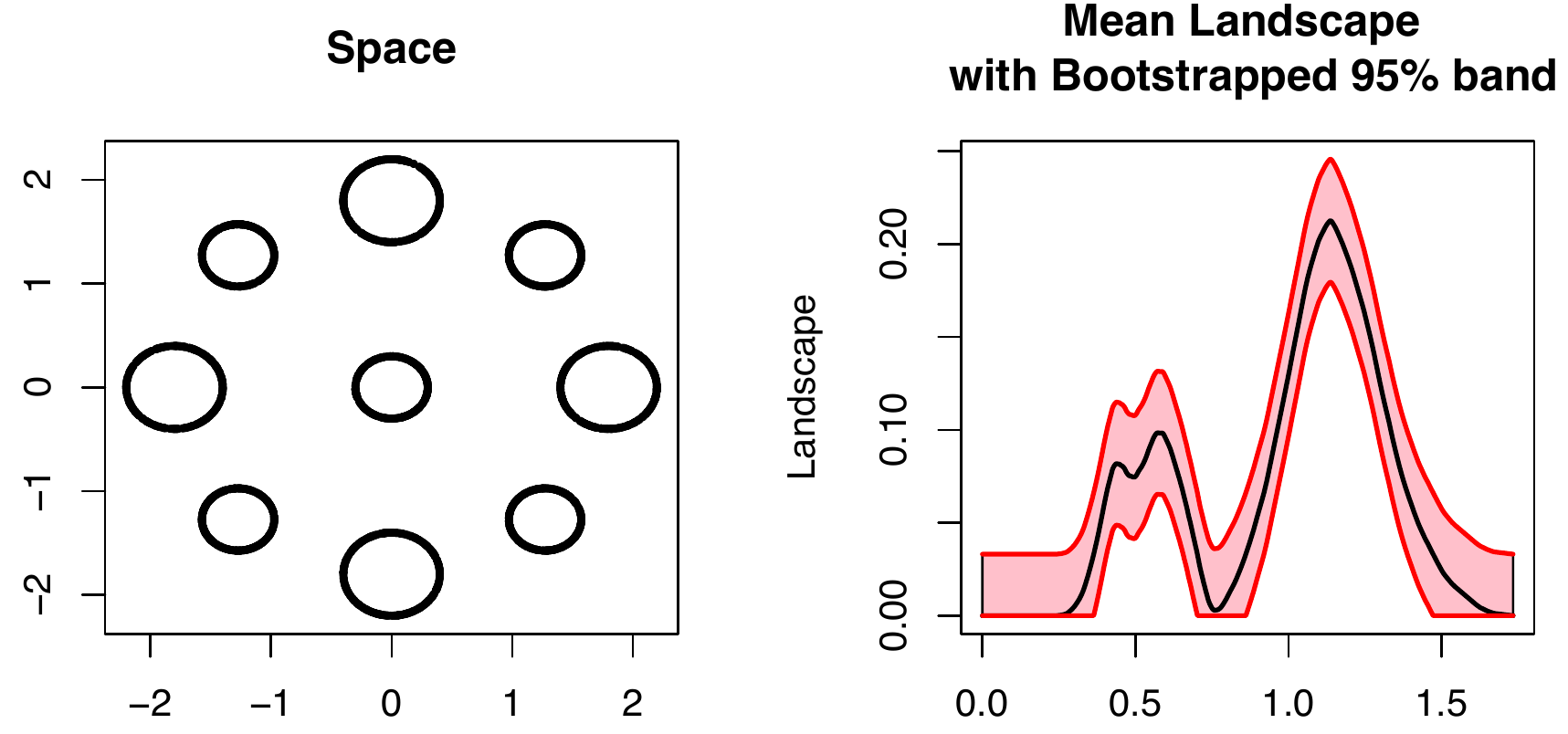}
 \caption{Left: The set of circles from which samples are taken. Right: The confidence band for
 the persistence landscape corresponding to the distance to the point set.}
 \label{fig:circles}
\end{figure}

%%%%%%%%%%%%%%%%%%%%%%%%%%%%%%%%%%%%%%%%%%%%%%%%%%%%%%%%%%%%%%%%%%%%%%%%
% Section: Discussion
%%%%%%%%%%%%%%%%%%%%%%%%%%%%%%%%%%%%%%%%%%%%%%%%%%%%%%%%%%%%%%%%%%%%%%%%
\subsection{Discussion}\label{sec:discussion}
% SUMMARY FUNCTIONS
In this paper, we have described the bootstrap as it
applies to persistence diagrams and landscapes.  The
purpose of this paper was to introduce the bootstrap
and the bootstrap empirical process to topologists.
In a related paper (\cite{arxiv_stathom}), 
aimed towards a statistical audience, we
derive the convergence rates for the technique
presented in~\secref{ss:dgmbootstrap}, as well as present three other
methods for computing confidence sets for persistence diagrams.

The persistence landscape can be thought of as
a summary function of a persistence diagram.  
The bootstrap method that we presented in~\secref{ss:landscapeboot}
trivially generalizes to handle all landscapes $\mathcal{L}(k,t)$.  
Furthermore, we need not limit the
scope of this method to landscape functions. 
In a future paper, we plan to investigate other
meaningful summary functions as well as the convergence
rates for the techniques presented in~\secref{ss:landscapeboot}.

We have demonstrated how the bootstrap works for two examples,
given in~\figref{fig:torus} and~\figref{fig:circles}.
Part of our ongoing research is investigating applications for
these confidence intervals; in particular, we are applying
it to real (rather than simulated) data sets.
One can use the confidence intervals for hypothesis
testing, but an open question is how to determine the power
of such a test.

\section*{Acknowledgement} The authors would like to thank Sivaraman Balakrishnan
for his insightful discussions.

\renewcommand{\refname}{References}
\bibliographystyle{plainnat}
\bibliography{Bootstrap}

\end{document}